\documentclass[11pt]{amsart}
\usepackage[a4paper]{geometry}
\usepackage{graphicx}
\usepackage{amsfonts}
\usepackage{amsmath}
\usepackage{amssymb}
\usepackage{amsthm}
\usepackage{newlfont}
\usepackage{yhmath}
\usepackage{color}

\newtheorem{pro}{Proposition}[section]
\newtheorem{thm}{Theorem}[section]
\newtheorem{cor}{Corollary}[section]

\newcommand{\e}{\mathbb{E}}
\newcommand{\p}{\mathbb{P}}

\allowdisplaybreaks

\usepackage{tikz}
\usetikzlibrary{positioning,shapes,arrows}
\usepackage{ytableau}  
\allowdisplaybreaks         
\ytableausetup{boxsize=1em}   

\begin{document}

\title{Transition Density of an Infinite-dimensional diffusion with the Jack Parameter
}

\author{Youzhou Zhou}
\address{Department of Pure Mathematics\\
Xi'an Jiaotong-Liverpool University \\
111 Renai Road\\
Suzhou, Jiangsu, China 215 123}

\email{youzhou.zhou@xjtlu.edu.cn}
\thanks{This research is supported by NSFC: 11701570.}



\begin{abstract}
From the Poisson-Dirichlet diffusions to the $Z$-measure diffusions, they all have explicit transition densities. In this paper, we will show that the transition densities of the $Z$-measure diffusions can also be expressed as a mixture of a sequence of probability measures on the Thoma simplex. The coefficients are the same as the coefficients in the Poisson-Dirichlet diffusions. This fact will be uncovered by a dual process method in a special case where the $Z$-measure diffusions is established through up-down chain in the Young graph. 
 \end{abstract}
 
 \date{\today}

\maketitle

 \section{Introduction} 

The Poisson-Dirichlet distribution $\mathrm{PD}(0,\theta),\theta>0,$ is proposed by Kingman \cite{King1} in the simplex $\overline{\nabla}_{\infty}=\{x\in[0,1]^{\infty}\mid x_{1}\geq x_{2}\geq\cdots\geq0,\sum_{i=1}^{\infty}x_i\leq 1\}$. Pitman generalized the one-parameter Poisson Dirichlet distribution $\mathrm{PD}(0,\theta)$ to the two-parameter Poisson-Dirichlet distribution $\mathrm{PD}(\alpha,\theta),\alpha\in(0,1),\theta+\alpha>0$, in \cite{EP}. They are also the representing measure of the Ewens-Pitman partition structure. Partition structure coined by Kingman in \cite{King} is an exchangeable partition distribution $M$ of the set $\mathbb{N}=\{1,2,\cdots,n,\cdots\}$. Due to its exchangeability, the restriction of $M$ to $\{1,2,\cdots,n\}$ is $M_n(\eta)=\dim(\eta)\varphi(\eta)$, where $\eta=(\eta_1,\cdots,\eta_l), |\eta|=:\sum_{i=1}^l\eta_i=n,$ is an integer partition of $n$. Moreover, $\varphi(\eta)$ is the probability of a single partition with cluster sizes $\eta$. Due to exchangeability, two partitions have the same probability as long as their cluster sizes are the same. So $\dim(\eta)$ is the multiplicities of such partitions with cluster size $\eta$. Then the family of distributions $\{M_{n},n\geq1\}$ will satisfy a natural consistent condition and they are uniquely determined by a representing measure in the Kingman simplex $\overline{\nabla}_{\infty}$ due to its exchangeability and the de Finetti theorem. 

Let $\Gamma_n$ be the totality of integer partitions of $n$, then $\Gamma=\cup_{n\geq0}^{\infty}\Gamma_n$ exhausts all integer partitions, and $\Gamma_{0}=\emptyset$ is treated as an empty partition. Usually $\Gamma$ can be the vertex set of a graded branching diagram $(\Gamma,\chi)$ where $\chi(\eta,\omega)$ assigns positive weight to an edge joining $\eta$ and $\omega$, and $0$ otherwise (see Figure \ref{bd}). Each integer partition is represented by a Young diagram (see Figure \ref{yd}), we say $\eta\subset\zeta$ if the Young diagram of $\eta$ is contained in the Young diagram of $\zeta$. There are various kinds of edge weights. In the algebra $\cal{A}$ of  symmetric functions with variables $(x_1,\cdots,x_n,\cdots)$, there are various kinds of linear bases $\{f_{\eta},\eta\in\Gamma\}$(refer to \cite{Mac} or Appendix), such as the monomial functions, the Shur functions and the Jack functions. These functions usually satisfy the Pieri formula 
$$
\left(\sum_{i=1}^{\infty}x_i\right)f_{\eta}=\sum_{\eta\subset\zeta, |\eta|+1=|\zeta|}\chi(\eta,\zeta)f_{\zeta},
$$
where $\eta$ can be obtained from $\zeta$ by removing a box. The edge weights are chosen to be the coefficients in the above Pieri formula. The weights from the monomial functions, the Shur functions and the Jack functions define the Kingman graph, the Young graph and the Jack graph respectively. The specialization of the symmetric functions can be regarded as an algebra homomorphism $\Phi: \cal{A}\to \mathbb{R}$. In particular, a specialization mapping base functions to positive values will determine a positive harmonic function $\varphi(\eta)=\Phi(f_{\eta})$ on $\Gamma$, satisfying 
$$
\varphi(\eta)=\sum_{\eta\subset\zeta, |\eta|+1=|\zeta|}\chi(\eta,\zeta)\varphi(\zeta).
$$ 
Under pointwise convergence topology, the space $\cal{C}$ of all positive harmonic functions becomes compact and convex. Therefore, one can expect that $\varphi(\eta)=\int_{E} K(\eta,x)\mu(dx)$ where $E$ is the Martin boundary of $\cal{C}$, $K(\eta,x)$ is the extremity and $\mu(dx)$ is a probability measure.  For the Kingman graph, this representation is the Kingman's one-to-one correspondence of the partition structures \cite{King} and $K(\eta,x)$ is the continuous extension of the monomial functions. For the Jack graph, the representation is also established in \cite{KOO}, and $K(\eta,x)$ is the extended Jack functions.  The $Z$-partition structure is a partition structure defined on the Jack graph, and its representing measure is called $Z$-measure, denoted as  $\mathcal{Z}(z,z^{\prime},\vartheta), z,z^{\prime}\in\mathbb{C},\vartheta>0$.

Petrov established the two-parameter Poisson-Dirichlet diffusion with the stationary distribution $\mathcal{PD}(\alpha,\theta)$ through an up-down Markov chain on the Kingman graph \cite{Pet}. 
 Similarly, Olshanski \cite{Ol_2} has established a reversible diffusion on the Thoma simplex
\begin{align*}
\Omega=\{(\alpha,\beta)\in[0,1]^{\infty}\times[0,1]^{\infty}\mid& \alpha_1\geq\alpha_2\geq\cdots\geq0,\\
&\beta_1\geq\beta_2\geq\cdots\geq0,\sum_{i=1}^{\infty}\alpha_i+\sum_{j=1}^{\infty}\beta_j\leq1\}.
\end{align*}
Its stationary distribution is the measure $\mathcal{Z}(z,z^{\prime},\vartheta)$. In this paper, we will use the $Z$-measure diffusion to refer to the diffusion obtained by Olshanski in \cite{Ol_2}. 

Interestingly, both types of diffusions have similar explicit transition densities, and the spectrums of their generators are also the same $\{0,\frac{n(n-1+\theta)}{2}\mid n\geq2\}$ where $\theta=\frac{zz^{\prime}}{\vartheta}>0$ for the $Z$-measure diffusion.

Recently, spectral expansion is applied to the $Z$-measure diffusion again  in \cite{SY_K} to derive its explicit transition density. Similar method has previously been used to establish the transition density of the two-parameter Poisson-Dirichlet diffusion in \cite{FSWX}. In fact, this method was first adopted by Ethier in \cite{ET}. Surprisingly, by rearranging the density in \cite{FSWX}, Zhou \cite{Zhou} yields the following expression 
\begin{equation}
p(t,x,y)=\widetilde{d_1^{\theta}}(t)+\sum_{n=2}^{\infty}d_n^{\theta}(t) p_{n}(x,y),\label{ex}
\end{equation}
 where $p_n(x,y)$ is a transition kernel and $\{\widetilde{d}_1^{\theta}(t),d_n^{\theta}(t),n\geq2\}, \theta>-1$ is the distribution of a pure death process related to the Kingman coalescent in \cite{Zhou}. 
  
  In this paper, we will rearrange the transition density of the $Z$-measure diffusion as the treatment in \cite{Zhou}. We obtain that the $Z$-measure diffusion has a similar expression 
 \begin{equation}
q(t,\sigma,\omega)=\widetilde{d_1^{\theta}}(t)+\sum_{n=2}^{\infty}d_n^{\theta}(t) \mathcal{K}_{n}(\sigma,\omega)\label{zf},
\end{equation}
 where $\mathcal{K}_{n}(\sigma,\omega)$ is also a transition kernel and $\{\widetilde{d}_1^{\theta}(t),d_n^{\theta}(t),n\geq2\}, \theta>0$ is the same as the coefficients in (\ref{ex}). 
 
 The next natural question would be why their transition densities have the same coefficients. This question has been resolved for the two-parameter Poisson-Dirichlet diffusion in \cite{GSMZ} by a dual process method.
 In order to understand why the $Z$-diffusion also has the similar coefficients in its density expression (\ref{zf}), we will also apply the dual process method to the $Z$-measure diffusion $Y_t$. But we can only find the dual process of the the $Z$-measure diffusion when $\vartheta=1$ because we relies on the equation (5) in \cite{Ol_2}. When $\vartheta\neq 1$, similar equation is not known. For $\vartheta=1$, the dual process $\mathcal{D}_{t}$ is also a partition-valued jump process characterized by the generator 
$$
 \mathcal{L}_1^{J}f(\eta)=-\frac{n(n-1+\theta)}{2}f(\eta)+\frac{n(n-1+\theta)}{2}\sum_{\zeta\subset\eta,|\eta|=|\zeta|+1} p_1^{\downarrow}(\eta,\zeta)f(\zeta),~ \mathcal{L}_1^{J}1=0,
$$
 where $p_1^{\downarrow}(\eta,\zeta)$ is the transition probability of the down Markov chain in the Jack graph discussed in \cite{Ol_2}, $f$ is the continuous function on the Thoma simplex. The dual relation is defined through a bivariate function $F(\eta, \omega)=\frac{s_{\eta}^{o}(\omega)}{\e_{z,z^{\prime},1}s_{\eta}^{o}}$ which is the normalized kernel in the representation of the Young graph. The notation $\e_{z,z^{\prime},1}$ should always be interpreted as the expectation with respect to the $Z$-measure $\mathcal{Z}(z,z^{\prime},1)$. We show that duality relation reads as $\e_{\omega}F(\eta, Z_t)=\e_{\eta}F(\mathcal{D}_{t}, \omega)$. Because the distribution of $\mathcal{D}_{t}$ is easier to calculate, then the expression (\ref{zf}) can be obtained. Now the radial process $|\mathcal{D}_{t}|$ of the dual process $\mathcal{D}_{t}$ will be exactly the same as that in the two-parameter Poisson-Dirichlet diffusion. This will eventually determine the coefficients in the density expression (\ref{zf}). Though when $\vartheta\neq1$, we can not find the find dual process, the result when $\vartheta=1$ encourage us to conjecture that the dual process of the $Z$-measure diffusion should also be a partition-valued jump process whose jump rate is $\frac{n(n-1+\theta)}{2}$ and its embedded chain is the down Markov chain in  \cite{Ol_2}.

 The plan of the paper is as follows. In section 2, we will introduce the branching diagrams and the up and down Markov chains. In section 3, we will talk about the $Z$-measure diffusion and its transition density. In section 4, we will use the dual process method to derive the transition density of the $Z$-measure diffusion when $\vartheta=1$. In the last section, a few conjectures will be presented. 

\section{Branching diagram and up-down Markov chain}
\subsection{Branching Diagram}
For $n,l\in\mathbb{Z}_{+}$, $\eta=(\eta_1,\cdots,\eta_l)\in\mathbb{N}^{l}$ is called an integer partition of $n$ if $\eta_1\geq\eta_2\geq\cdots\eta_l>0$ and $|\eta|:=\sum_{i=1}^{l}\eta_i=n$. Define $\alpha_i(\eta)=\#\{1\leq j\leq l\mid \eta_j=i\},1\leq i\leq n$, then $(\alpha_1(\eta),\cdots,\alpha_n(\eta))$ is a different representation of the partition $\eta$. Denote $\Gamma_n$ to be the set of all integer partitions of $n$, then $\Gamma=\cup_{n\geq0}\Gamma_n$ is the set of all integer partitions, where $\Gamma_0=\{\emptyset\}$ and $\emptyset$ is an empty partition. An integer partition $\eta=(\eta_1,\cdots,\eta_l)$ may also be represented by the Young diagram (Figure \ref{yd}) defined by attaching boxes at position $(i,j)$ where $1\leq i\leq l, 1\leq j\leq \eta_i$. Here the row number increases from top to bottom and the column number increases from left to right.
\begin{figure}[htbp]
\begin{center}
\ytableausetup{boxsize=1.5em}
\ydiagram{5,4,1}
\caption{Young diagram $\eta=(5,4,1)$}
\label{yd}
\end{center}
\end{figure}
In this paper, we will interchangeably use $\eta$ to represent either an integer partition or its Young diagram. We define the diagonal line of the Young diagram $\eta$ as the set of boxes $\{(i,i)\mid i\leq \eta_{i},\eta^{\prime}_{i}\}$ and $r$ to be the length of the diagonal line of $\eta$. The transposition of $\eta$ with respect to its diagonal line will give us a new partition $\eta^{\prime}$ called conjugate of $\eta$.  We say $\eta\subset \nu$ if the Young diagram of $\eta$ is contained in the Young diagram of $\nu$. Then $\subset$ is a partial order in $\Gamma$. For each box $b=(i,j)$ in the partition $\eta$, we will define its arm length as $a(b)=\eta_i-j$ and its leg length as $l(b)=\eta^{\prime}_j-i$. A partition $\eta$ may also be represented by its Frobenius coordinates $\widetilde{\eta}= (a(1,1)+\frac{1}{2},\cdots,a(r,r)+\frac{1}{2}; l(1,1)+\frac{1}{2},\cdots,l(r,r)+\frac{1}{2})$.

A branching diagram is a graded graph $(\Gamma,\chi)$ (Figure \ref{bd}), where $\Gamma=\cup_{n\geq0}\Gamma_n$ is the vertex set. There is an edge joining $\eta\in\Gamma_n$ and $\nu\in\Gamma_{n+1}$ if and only if $\eta\subset \nu$, and their edge weight is $\chi(\eta,\nu)$. 

\begin{figure}
\begin{center}
\ytableausetup
 {boxsize=0.5em}
\begin{tikzpicture}[scale=0.5]
      \node(0) at (-0.5,0) {$\emptyset$};
      \node(1) at (1,0) {\ydiagram{1}};
       \node(2a)at (2.5,0.7){\ydiagram{2}};
       \node(2b) at (2.5,-0.7){\ydiagram{1,1}};
       \node(3a) at (4.5, 1.5) {\ydiagram{3}};
      \node(3b) at (4.5,0) {\ydiagram{2,1}};
       \node(3c)at (4.5,-1.5){\ydiagram{1,1,1}};
       \node(4a) at (7.5,3){\ydiagram{4}};
       \node(4b) at (7.5,1.5) {\ydiagram{3,1}};
      \node(4c) at (7.5,0) {\ydiagram{2,2}};
       \node(4d)at (7.5,-1.4){\ydiagram{2,1,1}};
     \node(4e) at (7.5,-3){\ydiagram{1,1,1,1}};
     \path(0) edge (1);
   \path(1) edge (2a);
    \path(1) edge (2b);
     \path(2a) edge (3a);
    \path(2a) edge (3b);
     \path(2b) edge (3b);
    \path(2b) edge (3c);
     \path(3a) edge (4a);
    \path(3a) edge (4b);
     \path(3b) edge (4b);
    \path(3b) edge (4c);
    \path(3b) edge (4d);
    \path(3c) edge(4d);
    \path(3c) edge(4e);
    \end{tikzpicture}
\caption{Branching Diagram}
\label{bd}
\end{center}
\end{figure}

When the edge weights are determined by the coefficients in the Pieri formula of the monomial symmetric functions, we will have the Kingman graph. Its edge weight is defined as $\chi^{K}(\eta,\zeta)=\alpha_{\zeta_{i}}(\zeta)$ if $\zeta-\eta=(i,j)$, i.e. $\zeta$ can be obtained from $\eta$ by attaching a box at $(i,j)$.  If we choose the edge weight to be the coefficients in the Pieri formula of the Jack functions, then we end up with the Jack graph, whose edge weight, denoted as $\chi^{J}_{\vartheta}(\eta,\zeta)$, will be  
$$
\chi^{J}_{\vartheta}(\eta,\zeta)=\prod_{b}\frac{(a(b)+(l(b)+2)\vartheta)(a(b)+1+l(b)\vartheta)}{(a(b)+1+(l(b)+1)\vartheta)(a(b)+(l(b)+1)\vartheta)}
$$
where $b$ runs over all boxes in the $j$-th column of $\eta$ if $\zeta-\eta=(i,j)$. When $\vartheta=1$, then $\chi^{J}_{1}=1$ and the Jack graph reduces to the Young graph.  

We define the weight of a path $t: \eta=\eta(0)\subset\cdots\subset\eta(k)=\zeta$ as $\prod_{i=0}^{k-1}\chi(\eta(i),\eta(i+1))$. Then we define the total weight between $\eta$ and $\zeta$ in the Kingman graph as $\dim^K(\eta,\zeta)=\sum_{\eta=\eta(0)\subset\cdots\subset\eta(k)=\zeta}\prod_{i=0}^{k-1}\chi(\eta(i),\eta(i+1))$. Similarly, the total weight between $\eta$ and $\zeta$ in the Jack graph will be $\dim^J_{\vartheta}(\eta,\zeta)=\sum_{\eta=\eta(0)\subset\cdots\subset\eta(k)=\zeta}\prod_{i=0}^{k-1}\chi(\eta(i),\eta(i+1))$.
In particular, when $\eta=\emptyset$, we regard $\dim^K(\nu)=\dim(\emptyset,\nu)$ as the total weight of $\nu$ in the Kingman graph and $\dim^J_{\vartheta}(\nu)=\dim^{J}_{\vartheta}(\emptyset,\nu)$ as the total weight of $\nu$ in the Jack graph. Naturally, we have 
\begin{equation}\label{rw}
\dim^{K}(\zeta)=\sum_{\eta\subset\zeta,|\zeta|=|\eta|+1}\chi^J(\eta,\zeta)\dim^K(\eta)
\end{equation}
and
\begin{equation}\label{jw}
\dim^{J}_{\vartheta}(\zeta)=\sum_{\eta\subset\zeta,|\zeta|=|\eta|+1}\chi^J(\eta,\zeta)\dim^J_{\vartheta}(\eta).
\end{equation}

In the Kingman graph, the total weight of $\eta$ is $\dim^{K}(\eta)=\frac{n!}{\eta_1!\cdots\eta_l!}$ if $\eta=(\eta_1,\cdots,\eta_l)\in\Gamma_n$. In the Jack graph, the total of $\eta$ is $\dim^{J}_{\vartheta}(\eta)=\frac{n!}{H(\eta;\vartheta)H^{'}(\eta;\vartheta)}$, where $H(\eta;\vartheta)=\prod_{b\in\eta}(a(b)+1+l(b)\eta), H^{'}(\eta;\vartheta)=\prod_{b\in\eta}(a(b)+(1+l(b))\eta)$. When $\vartheta=1$, the total weight of $\eta$ is $\dim_{1}^{J}(\eta)=\frac{n!}{H^2(\eta;1)}$ in the Young graph. 

\subsection{Down-Up Markov Chain} Due to the equation (\ref{rw}), one can easily construct a down Markov chain $\{D_{n}^{K},n\geq1\}$ in the Kingman graph with the following transition probability 
$$
p^{\downarrow,K}(\zeta,\eta)=\frac{\chi^K(\eta,\zeta)\dim^K(\eta)}{\dim^K(\zeta)}=\frac{\alpha_{\zeta_i}(\zeta)\zeta_i}{n}, ~\zeta-\eta=(i,j),\zeta\in\Gamma_n.
$$
This down Markov chain is the embedded chain of the dual process in \cite{GSMZ}.
Similarly, one can also construct a down Markov chain $\{D_{n}^{J,\vartheta},n\geq1\}$ in the Jack graph with the following transition probability 
$$
p^{\downarrow,J}_{\vartheta}(\zeta,\eta)=\frac{\chi^J_{\vartheta}(\eta,\zeta)\dim^J_{\vartheta}(\eta)}{\dim^J_{\vartheta}(\zeta)}, ~\zeta\in\Gamma_n, \eta\in\Gamma_{n-1}.
$$
When $\vartheta=1$, this down Markov chain is the embedded chain of the dual process of the $Z$-measure diffusion that we are going to discuss in Section \ref{s4}. 

As you may see, the down Markov chain depends only on the edge weights of the graph. But we can also construct up Markov chain if we have exchangeable partition structures in the branching diagram. The Ewens-Pitman partition structure can be used to construct an up Markov chain  $\{U_n^{K,\theta,\alpha},n\geq1\}$ in the Kingman graph \cite{Pet}.  Then we can define an up-down chain $\{UD^{K,\theta,\alpha,n}_{m},m\geq1\}$ on $\Gamma_n$, updating itself as Gibbs sampler,
\begin{align}
&\p(UD^{K,\alpha,\theta,n}_{m+1}=\zeta\mid UD^{K,\alpha,\theta,n}_{m}=\eta)\nonumber\\
=&\sum_{|\nu|=n+1}\p(U_n^{K,\theta,\alpha}=\nu\mid U_n^{K,\theta,\alpha}=\eta)\p(D_n^{K}=\zeta\mid D_n^{K}=\nu)\label{gibbs}
\end{align}
Naturally, the partition distribution $M_n^{K,\theta,\alpha}$ will serve as the stationary distribution of this up-down chain. The usual space and time scaling yields the two-parameter Poisson-Dirichlet diffusion in \cite{Pet}. The Ewens-Pitman partition structure can be replicated by the Blackwell-MacQueen urn model, and it has found many applications in classification problems through Bayesian statistics \cite{BE}. 

For the Jack graph there is also a special $Z$ partition structure \cite{BO1}
\begin{equation}\label{zp}
M^{J,z,z^{\prime},\vartheta}_n(\eta)=\dim^{J}_{\vartheta}(\eta)\frac{(z)_{\eta,\vartheta}(z^{\prime})_{\eta,\vartheta}}{\theta_{(n)}H^{\prime}(\eta;\vartheta)}, ~\theta=\frac{zz^{\prime}}{\vartheta},\vartheta>0,
\end{equation}
where $z,z^{\prime}$ are either (i) $z\in\mathbb{C}-(\mathbb{Z}_{\leq0}+\vartheta \mathbb{Z}_{\geq0})$ and $z^{\prime}=\overline{z}$ (principal case)or (ii) $\vartheta$ is rational number and $z,z^{\prime}$ are real numbers belonging to an interval between two consecutive lattice points in $\mathbb{Z}+\vartheta\mathbb{Z}$(complementary case).
Here $(z)_{\eta;\vartheta}=\prod_{(i,j)\in\eta}(z+(j-1)-(i-1)\vartheta)$. The representing measure $\mathcal{Z}(z,z^{\prime},\vartheta)$ of the partition structure (\ref{zp}) is the $Z$-measure. Similarly, one can construct an up Markov chain $\{U^{J,z,z^{\prime},\vartheta}_{n},n\geq1\}$ \cite{Ol_2}, By mimicking the update in (\ref{gibbs}), one can also construct an up-down Markov chain $\{UD^{J,z,z^{\prime},\vartheta,n}_{m},m\geq1\}$ on $\Gamma_n$. Then the usual space and time scaling yields the $Z$-diffusion in \cite{Ol_2}.
As far as the author's knowledge, no quick replication of $Z$ partition structure $M^{J,z,z^{\prime},\vartheta}_n(\eta)$ has been spotted now. So whether the $Z$ partition structure can be applied to classification problems is still open.

\section{The $Z$-measure diffusion and its transition density}

\subsection{The $Z$-measure diffusion}
The diffusion approximation of the up-down Markov chain $\{UD^{J,z,z^{\prime},\vartheta,n}_{m},m\geq1\}$ on $\Gamma_n$ can be carried out by the following space scaling 
$$
\pi: \eta\to \frac{\widetilde{\eta}}{n}=\left(\frac{a(1,1)+\frac{1}{2}}{n},\cdots,\frac{a(r,r)+\frac{1}{2}}{n}; \frac{l(1,1)+\frac{1}{2}}{n},\cdots,\frac{l(r,r)+\frac{1}{2}}{n}\right),
$$ 
where $\widetilde{\eta}$ is the Frobenius coordinates of $\eta$. 
As $n\to\infty$, Olshanski in \cite{Ol_2} has shown that $\pi(UD^{J,z,z^{\prime},\vartheta,n}_{[n^2t]})$ converges to the $Z$-measure diffusion $Y_t$ on the Thoma simplex $\Omega$ with the following pre-generator  
  \begin{align*}
 A_{z,z^{\prime},\vartheta}=&\frac{1}{2}\sum_{i,j\geq2}ij(\varphi_{i+j-1}^{o}-\varphi_{i}^{o}\varphi_j^{o})\frac{\partial^2}{\partial \varphi_i^{o}\partial \varphi_j^{o}}+\frac{\vartheta}{2}\sum_{i,j\geq1}(i+j+1)\varphi_i^{o}\varphi_j^{o}\frac{\partial}{\partial \varphi_{i+j+1}^{o}}\\
 &+\frac{1}{2}\sum_{i\geq2}\left[(1-\vartheta)i(i-1)\varphi_{i-1}^{o}+(z+z^{'})i\varphi_{i-1}^{o}-i(i-1)\varphi_i^{o}-i\frac{zz^{'}}{\vartheta}\varphi_i^{o}\right]\frac{\partial}{\partial \varphi_i^{o}}.
 \end{align*}
The core of $ A_{z,z^{\prime},\vartheta}$ is spanned by $\{\varphi_j^{o},j\geq1\}$, where $\varphi_j^{o}$ is the image of $\varphi_j(x)=\sum_{i=1}^{\infty}x_i^j$ under the special algebra homomorphisms $\Phi_{\omega}: \mathcal{A}\to\mathbb{R}, \omega=(\alpha,\beta)\in\Omega,$ 
$$
\Phi_{\omega}(\varphi_{n})=\sum_{i=1}^{\infty}\alpha_i^{n}+(-\vartheta)^{n-1}\sum_{j=1}^{\infty}\beta_j^{n}, ~n\geq2, ~\Phi_{\omega}(\varphi_{1})=1. 
$$
Since the Jack functions $J_{\zeta}(x;\vartheta)$ can be written as linear combinations of $\varphi_{\eta_1}\cdots\varphi_{\eta_l}$, where $\eta=(\eta_1,\cdots,\eta_l)\subset \zeta$, then $J_{\zeta}^{o}(\omega;\vartheta):=\Phi_{\omega}(J_{\zeta}(x;\vartheta))$. 

In particular, when $\vartheta=1$, we will have the extended Shur function $s^{o}_{\zeta}(\omega):=\Phi_{\omega}(J_{\zeta}(x;1))$. The $Z$-measure diffusion reduces to the diffusion in \cite{BO2}.
 
 \subsection{Transition density of the $Z$-measure diffusion $Y_t$}
  By spectral expansion of $A_{z,z^{\prime},\vartheta}$ in the Hilbert space $L^2(\Omega, \mathcal{Z}(z,z^{\prime},\vartheta))$, Korotkikh obtained the following explicit transition density of $Y_t$ in \cite{SY_K}. 
 \begin{pro}
The transition density of $Y_t$ is 
\begin{equation}\label{spec_expan}
q(t,\sigma,\omega)=1+\sum_{m=2}^{\infty}e^{-t\lambda_m}G_m(\sigma,\omega),
\end{equation}
where $\theta=\frac{zz^{\prime}}{\vartheta}$, $\lambda_m=\frac{m(m-1+\theta)}{2}$,
$
G_m(\sigma,\omega)=\sum_{n=0}^{m}(-1)^{m-n}\binom{m}{n}\frac{(\theta+2m-1)(\theta+n)_{m-1}}{m!}\mathcal{K}_n(\sigma,\omega),
$
and
 \begin{equation}\label{kernel}
 \mathcal{K}_n(\sigma,\omega)=\sum_{|\eta|=n}\frac{j_{\eta}(\sigma;\vartheta)j_{\eta}(\omega;\vartheta)}{\mathbb{E}_{z,z^{\prime},\vartheta}j_{\eta}},~ j_{\eta}(\omega;\vartheta)=\dim_{\vartheta}^{J}(\eta)J^o_{\eta}(\omega;\vartheta).
 \end{equation}
 In particular, when $n=0,1$, $\mathcal{K}_n(\sigma,\omega)=1$.
 \end{pro}
 
 In \cite{SY_K}, 
$
 K_n^{o}(\sigma,\omega)=\sum_{|\eta|=n}\frac{H^{\prime}(\eta;\vartheta)}{H(\eta;\vartheta)}\frac{J^{o}(\sigma;\vartheta)J^{o}(\omega;\vartheta)}{(z)_{\eta,\vartheta}(z^{\prime})_{\eta,\vartheta}}.
$ 
By Proposition \ref{jack}, Borodin and Olshanski \cite{BO1} have shown that the representing measure of the $Z$ partition structure is the $Z$ measure $\mathcal{Z}(z,z^{\prime},\vartheta)$ .
\begin{pro}\label{jack}
For a partition structure $\{M_{n},n\geq1\}$ on the Jack graph, there is a unique probability measure $\mu$ on the Thoma simplex $\Omega$ such that 
$$
M_n(\eta)=\int_{\Omega}j_{\eta}(\omega;\vartheta)\mu(d\omega),
$$
Moreover, $\mu_n(d\omega)=\sum_{\eta\in\Gamma_n}M_n(\eta)\delta_{\frac{\widetilde{\eta}}{n}}(d\omega)$ will converge weakly to $\mu$. Here 
$$
\widetilde{\eta}=\left(a(1,1)+\frac{1}{2},\cdots,a(r,r)+\frac{1}{2}; l(1,1)+\frac{1}{2},\cdots,l(r,r)+\frac{1}{2}\right)
$$ 
is the Frobenius coordinate of $\eta$.
\end{pro}

Therefore, similar to the Ewens sampling formula, we have the following sampling formula
 \begin{equation}\label{sap}
 \mathbb{E}_{z,z^{\prime},\vartheta}j_{\eta}=M_{n}^{z,z^{\prime},\vartheta}(\eta)
 \end{equation}
 and $K_n^{o}(\sigma,\omega)=\frac{\mathcal{K}_n(\sigma,\omega)}{n!(\theta)_{n}}$.  Then one can easily see that (\ref{spec_expan}) is equivalent to the density in \cite{SY_K}.

 In this paper, we will rearrange the right hand side of the equation (\ref{spec_expan}) to yield a new representation
 \begin{thm}\label{main}
The transition density of $Y_t$ is 
\begin{equation}\label{con}
q(t,\sigma,\omega)=d_0^{\theta}(t)+d_1^{\theta}(t)+\sum_{n=2}^{\infty}d_n^{\theta}(t)\mathcal{K}_n(\sigma,\omega)
\end{equation}
where
\begin{align*}
d_{0}^{\theta}(t)=&1-\sum_{m=1}^{\infty}\frac{2m-1+\theta}{m!}(-1)^{m-1}\theta_{(m-1)}e^{-\lambda_{m}t}\\
d_{n}^{\theta}(t)=&\sum_{m=1}^{\infty}\frac{2m-1+\theta}{m!}(-1)^{m-n}\binom{m}{n}(n+\theta)_{(m-1)}e^{-\lambda_{m}t},n\geq1.
\end{align*}
 \end{thm}
Due to the estimation of $G_m$ in \cite{SY_K},  the proof of Theorem \ref{main} is the same as the proof of Theorem 2.1 in \cite{Zhou}. 
\begin{proof}
By Proposition 16 in \cite{SY_K}, we know there exists positive constants $c,d$ such that 
$$
\sup_{\sigma,\omega\in\Omega}|G_m(\sigma,\omega)|\leq c m^{dm}.
$$
Then one can show that the series in (\ref{spec_expan}) is uniformly convergent( see \cite{Zhou} or \cite{SY_K}). Then we can switch the order of summation in (\ref{spec_expan}). 
\begin{align*}
q(t,\sigma,\omega)=&1+\sum_{m=2}^{\infty}e^{-\lambda_{m}t}\Big[\sum_{n=2}^{m}(-1)^{m-n}\binom{m}{n}\frac{(\theta+2m-1)(\theta+n)_{m-1}}{m!}\mathcal{K}_n(\sigma,\omega)\\
&+\frac{2m+\theta-1}{m!}(-1)^{m-1}(\theta+1)_{(m-1)}m\mathcal{K}_1(\sigma,\omega)\\
&+\frac{2m+\theta-1}{m!}(-1)^{m}(\theta)_{(m-1)}\mathcal{K}_0(\sigma,\omega)\Big]
\end{align*}
Since $\mathcal{K}_1(\sigma,\omega)=\mathcal{K}_0(\sigma,\omega)=1$, we have 
\begin{align*}
q(t,\sigma,\omega)=&1+\sum_{m=2}^{\infty}e^{-\lambda_{m}t}\sum_{n=2}^{m}(-1)^{m-n}\binom{m}{n}\frac{(\theta+2m-1)(\theta+n)_{m-1}}{m!}\mathcal{K}_n(\sigma,\omega)\\
&+\sum_{m=2}^{\infty}e^{-\lambda_{m}t}\Big[\frac{2m+\theta-1}{m!}(-1)^{m-1}(\theta+1)_{(m-1)}m\\
&+\frac{2m+\theta-1}{m!}(-1)^{m}(\theta)_{(m-1)}\Big]\\
=&1+\sum_{m=2}^{\infty}e^{-\lambda_{m}t}\sum_{n=2}^{m}(-1)^{m-n}\binom{m}{n}\frac{(\theta+2m-1)(\theta+n)_{(m-1)}}{m!}\mathcal{K}_n(\sigma,\omega)\\
&+\sum_{m=2}^{\infty}e^{-\lambda_{m}t}\Big[\frac{2m+\theta-1}{m!}(-1)^{m-1}(\theta+1)_{(m-1)}m\\
&+\frac{2m+\theta-1}{m!}(-1)^{m}(\theta)_{(m-1)}\Big]\\
=&1+\sum_{m=2}^{\infty}e^{-\lambda_{m}t}\sum_{n=2}^{m}(-1)^{m-n}\binom{m}{n}\frac{(\theta+2m-1)(\theta+n)_{(m-1)}}{m!}\mathcal{K}_n(\sigma,\omega)\\
&+\sum_{m=2}^{\infty}e^{-\lambda_{m}t}\frac{2m+\theta-1}{m!}(-1)^{m-1}\Big[(\theta+1)_{(m-1)}m-(\theta)_{(m-1)}\Big]
\end{align*}
Note that $(\theta+1)_{(m-1)}m-(\theta)_{(m-1)}=0$ for $m=1$. Therefore,
\begin{align*}
q(t,\sigma,\omega)
=&1+\sum_{m=2}^{\infty}e^{-\lambda_{m}t}\sum_{n=2}^{m}(-1)^{m-n}\binom{m}{n}\frac{(\theta+2m-1)(\theta+n)_{(m-1)}}{m!}\mathcal{K}_n(\sigma,\omega)\\
&+\sum_{m=1}^{\infty}e^{-\lambda_{m}t}\frac{2m+\theta-1}{m!}(-1)^{(m-1)}\Big[(\theta+1)_{(m-1)}m-(\theta)_{(m-1)}\Big]\\
=&1-\sum_{m=1}^{\infty}e^{-\lambda_{m}t}\frac{2m+\theta-1}{m!}(-1)^{(m-1)}(\theta)_{(m-1)}\\
+&\sum_{m=1}^{\infty}e^{-\lambda_{m}t}\frac{2m+\theta-1}{m!}(-1)^{(m-1)}(\theta+1)_{(m-1)}m\\
+&\sum_{m=2}^{\infty}e^{-\lambda_{m}t}\sum_{n=2}^{m}(-1)^{m-n}\binom{m}{n}\frac{(\theta+2m-1)(\theta+n)_{(m-1)}}{m!}\mathcal{K}_n(\sigma,\omega)\\
=&d_0^{\theta}(t)+d_1^{\theta}(t)+\sum_{n=2}^{\infty} \mathcal{K}_n(\sigma,\omega)\sum_{m=n}^{\infty}e^{-\lambda_{m}t}(-1)^{m-n}\binom{m}{n}\frac{(\theta+2m-1)(\theta+n)_{(m-1)}}{m!}\\
=&\widetilde{d_1^{\theta}}(t)+\sum_{n=2}^{\infty}d_{n}^{\theta}(t)\mathcal{K}_n(\sigma,\omega)
\end{align*}
\end{proof}

\begin{cor}
The diffusion $Y_t$ satisfies the following ergodic inequality
\begin{align*}
\sup_{\omega\in \Omega}\|\p_{\omega}(Y_t\in\cdot)-\mathcal{Z}(z,z^{\prime},\vartheta)(\cdot)\|_{\mathrm{Var}}\leq \frac{(\theta+1)(\theta+2)}{2}e^{-(\theta+1)t},~\theta=\frac{zz^{\prime}}{\vartheta}>0.
\end{align*}
\end{cor}
This inequality can be easily derived from an inequality of tail probabilities (see \cite{Tav})
$$
\sum_{n=2}^{\infty}d_n^{\theta}(t)\leq  \frac{(\theta+1)(\theta+2)}{2}e^{-(\theta+1)t}.
$$

\section{Dual process of the $Z$-measure diffusion when $\vartheta=1$}\label{s4}

The transition density of the $Z$-measure diffusion looks so similar to the transition density of the two-parameter Poisson-Dirichlet diffusion. Given their huge differences, it is surprising that their transition densities are both mixture of distributions with exactly the same coefficients. To figure out why they have the same coefficients, we adopt the dual process method used in \cite{GSMZ}. The duality between $Y_t$ and its dual process $\mathcal{D}_{t}$ is defined through a bivariate function $F(\eta,\omega)$ where $\eta\in\Gamma$ and $\omega\in\Omega$. The bivariate function $F(\eta,\omega)$ is usually chosen to be the normalized kernel $F(\eta,\omega)=\frac{j_{\eta}(\omega;\vartheta)}{\e_{z,z^{\prime},\vartheta}j_{\eta}}$ in Proposition \ref{jack}. Then $Y_t$ and its dual $\mathcal{D}_t$ satisfy 
\begin{equation}\label{dual}
\e_{\eta}^*F(\mathcal{D}_t,\omega)=\e_{\omega}F(\eta, Y_t)
\end{equation}
where $\e_{\eta}^*$ is the expectation with respect to the distribution of $\mathcal{D}_t$. In this paper, we use the dual equation (\ref{dual}) to derive the transition density (\ref{con}) directly when $\vartheta=1$. This derivation will clearly explain why the coefficients $d_n^{\theta}(t)$ show up in the transition density (\ref{con}). However, when $\vartheta\neq1$, we fail to verify the dual process of $Y_t$ because we don't know whether the $Z$-measure diffusion has similar equation (5) in \cite{Ol_2}. 

\subsection{Dual process} In this section, we will consider test functions $g_{\eta}(\omega)=\frac{j_{\eta}(\omega;1)}{\e_{z,z^{\prime},\vartheta}j_{\eta}}$, so $\e_{z,z^{\prime},\vartheta}g(\eta)=1$. Because when $\vartheta=1$, $J_{\eta}(x;1)$ is just the Shur function $s_{\eta}(x)$. So $j_{\eta}(\omega;1)$ is  $\dim_{1}^{J}(\eta)\Phi_{\omega}(s_{\eta})=\dim_{1}^{J}(\eta)s^{o}_{\eta}(\omega)$. Moreover, by the equation (\ref{sap}), we know
$$
\e_{z,z^{\prime},1}\dim_{1}^{J}(\eta)s^{o}_{\eta}=\frac{n!}{H^2(\eta;1)}\frac{(z)_{\eta,1}(z^{\prime})_{\eta,1}}{(\theta)_{(n)}}
$$
We define $F(\eta,\omega)=g_{\eta}(\omega), \eta\in\Gamma,\omega\in\Omega$. So 
\begin{equation}\label{sam-formula}
F(\eta,\omega)=g_{\eta}(\omega)=\frac{H(\eta;1)\theta_{(n)}s_{\eta}^{o}(\omega)}{(z)_{\eta,1}(z^{\prime})_{\eta,1}}.
\end{equation}
Now consider a jump process $\mathcal{D}_t$ defined by 
 \begin{equation}
 \mathcal{L}^{J}_{1}f(\eta)=-\frac{n(n-1+\theta)}{2}f(\eta)+\frac{n(n-1+\theta)}{2}\sum_{\zeta\subset \eta,|\eta|=|\zeta|+1}\frac{\dim_{1}^{J}(\zeta)}{\dim^{J}_{1}(\eta)}f(\zeta),
 \end{equation}
 where $\frac{\dim_{1}^{J}(\zeta)}{\dim^{J}_{1}(\eta)}=p^{\downarrow,J}_{1}(\eta, \zeta)$ is the transition probability of the down Markov chain $D^{J,1}_n$ in the Young graph. Moreover $ \mathcal{L}^{J}_{1}1=0$, and because $s_{(1)}^{o}=1$ then $\mathcal{D}_t$ will be absorbed at state $\eta=(1)$. The radial process $|\mathcal{D}_t|$ of $\mathcal{D}_t$ is only determined by the jump rates $\frac{n(n-1+\theta)}{2},n\geq1$. Therefore, $|\mathcal{D}_t|$ is exactly the radial process in the Kingman coalescent by collapsing state $0$ and $1$ as a new state $1$. The distribution of $|\mathcal{D}_t|$ is obtained in \cite{Tav} and \cite{Zhou}. So as long as the jump rates are the same for the dual process, the coefficients in the their transition density, if exists, will be the same. 
\begin{thm}\label{d1}
 When $\vartheta=1$, the $Z$-measure diffusion $Y_t$ and $\mathcal{D}_t$ satisfy the following duality
\begin{equation}\label{dual-equation}
\e_{\eta}^*F(\mathcal{D}_t,\omega)=\e_{\omega}F(\eta, Y_t)
\end{equation}
 \end{thm}
\begin{proof}
By the Lemma 5.4 in \cite{BO2}, we know 
\begin{equation}\label{recur}
A_{z,z^{\prime},1}s_{\eta}^{o}(\omega)=-\frac{n(n-1+\theta)}{2}s_{\eta}^{o}(\omega)+\frac{1}{2}\sum_{\zeta\subset \eta, |\eta|=|\zeta|+1}\frac{(z)_{\eta,1}(z^{'})_{\eta,1}}{(z)_{\zeta,1}(z^{'})_{\zeta,1}}s^{0}_{\zeta}(\omega)
\end{equation}
Due to equation (\ref{sam-formula}), we know 
$$
s^{o}_{\eta}(\omega)=g_{\eta}(\omega)\frac{(z)_{\eta,1}(z^{'})_{\eta,1}}{H(\eta;1)\theta_{(n)}},~~s^{o}_{\zeta}(\omega)=g_{\zeta}(\omega)\frac{(z)_{\zeta,1}(z^{'})_{\zeta,1}}{H(\zeta;1)\theta_{(n-1)}}
$$
Replacing $s^{o}_{\eta}$ and $s^{o}_{\zeta}$ in equation (\ref{recur}) yields 
$$
A_{z,z^{\prime},1}g_{\eta}(\omega)=-\frac{n(n-1+\theta)}{2}g_{\eta}(\omega)+\frac{1}{2}\sum_{\zeta\subset \eta, |\eta|=|\zeta|+1}\frac{H(\eta;1)\theta_{(n)}}{H(\zeta;1)\theta_{(n-1)}}g_{\zeta}(\omega)
$$
Because $\frac{H(\eta;1)}{H(\zeta;1)}=n\frac{\dim_{1}^{J}(\zeta)}{\dim_{1}^{J}(\eta)}$ and $\frac{\theta_{(n)}}{\theta_{(n-1)}}=n-1+\theta$, we have
$$
A_{z,z^{\prime},1}g_{\eta}(\omega)=-\frac{n(n-1+\theta)}{2}g_{\eta}(\omega)+\frac{n(n-1+\theta)}{2}\sum_{\zeta\subset \eta, |\eta|=|\zeta|+1}\frac{\dim_{1}^{J}(\zeta)}{\dim_{1}^{J}(\eta)}g_{\zeta}(\omega).
$$
Therefore,
$$
A_{z,z^{\prime},1}F(\eta,\omega)=\mathcal{L}^{J}_{1}F(\eta,\omega).
$$
By Theorem 4.4.1 in \cite{EK}, one can show that the $Z$-measure diffusion $Y_t$ and $\mathcal{D}_t$ satisfy the dual equation (\ref{dual-equation}). 
\end{proof}

\begin{pro}\label{dual-prob1}
The dual process $\mathcal{D}_t$ has the following transition probability
$$
\p(\mathcal{D}_t=\eta\mid \mathcal{D}_0=\nu)=d_{mn}^{\theta}(t)\mathcal{H}(\eta,\nu),~\eta\in\Gamma_{n}, \nu\in \Gamma_m, n\leq m.
$$
Here 
$$
\mathcal{H}(\eta,\nu)=\frac{\dim^{J}_{1}(\eta)\dim^{J}_{1}(\eta,\nu)}{\dim^{J}_{1}(\nu)}
$$
and
\begin{equation}\nonumber
d_{mn}^\theta(t) =
\sum_{k=n}^{m}e^{-\frac{1}{2}k(k+\theta-1)t}
(-1)^{k-n}
\frac{(2k+\theta-1)(n+\theta)_{(k-1)}}{n!(k-n)!}
\frac{ m_{[k]} }{ (\theta+m)_{(k)} }.
\end{equation}
\end{pro}
\begin{proof} By the definition of $|\mathcal{D}_{t}|$ and Proposition 2.1 in \cite{Zhou}, we know 
$$
\p(|\mathcal{D}_{t}|=n\mid|\mathcal{D}_{0}|=m)=d_{mn}^{\theta}(t).
$$
Then 
\begin{align*}
\mathbb{P}_{\nu}(\mathcal{D}_t=\eta)=&\p(\mathcal{D}_{t}=\eta\mid \mathcal{D}_0=\nu)=\mathbb{P}(\mathcal{D}_t=\eta|  |\mathcal{D}_t|=n,\mathcal{D}_0=\nu)\mathbb{P}(|\mathcal{D}_{t}|=n\mid \mathcal{D}_0=\nu)\\
=&d_{mn}^{\theta}(t)\mathbb{P}(\mathcal{D}_t=\eta|  |\mathcal{D}_t|=n,\mathcal{D}_0=\nu).
\end{align*}
For $\nu\in\Gamma_{m}, \eta\in\Gamma_{n}$, there are paths of length $k=m-n$ joining $\nu$ and $\eta$. These paths are the realizations of the embedded down Markov chain $D^{J,1}_n$. Thus, 
\begin{align*}
\mathbb{P}(\mathcal{D}_t=\eta|  |\mathcal{D}_t|=n,\mathcal{D}_0=\nu)=&\sum_{\eta=\eta(k)\subset\cdots\subset \eta(0)=\nu}\prod_{i=0}^{k-1}p^{\downarrow,J}_{1}(\eta(i),\eta(i+1))\\
=&\sum_{\eta=\eta(k)\subset\cdots\subset \eta(0)=\nu}\prod_{i=0}^{k-1}\frac{\dim_{1}^{J}(\eta(i+1))\chi^{J}_{1}(\eta(i),\eta(i+1))}{\dim_{1}^{J}(\eta(i))}\\
=&\frac{\dim_1^{J}(\eta(k))}{\dim_{1}^{J}(\eta(0))}\sum_{\eta=\eta(k)\subset\cdots\subset \eta(0)=\nu}\prod_{i=0}^{k-1}\chi^{J}_{1}(\eta(i),\eta(i+1))\\
=&\frac{\dim_1^{J}(\eta)}{\dim_{1}^{J}(\nu)}\dim_{1}^{J}(\eta,\nu).
\end{align*}
\end{proof}

\begin{pro}\label{dual-prob2}
For $j_{\eta}(\omega;1), \eta\in\Gamma_m$, we have
\begin{equation}\label{dp2}
\e_{z,z^{\prime},1}[j_{\eta}(Y_t;1)\mid Y_0=\omega]=(\e_{z,z^{\prime},1}j_{\eta})\left[d_{m1}^{\theta}(t)+\sum_{n=2}^{m}d_{mn}^{\theta}(t)\sum_{\zeta\subset\eta, |\zeta|=n}\mathcal{H}(\eta,\zeta)\frac{j_{\zeta}(\omega;1)}{\e_{z,z^{'},1}j_{\zeta}}\right]
\end{equation}
\end{pro}
\begin{proof}
By the duality equation (\ref{dual-equation}), we know
\begin{align*}
\e_{\omega}g_{\eta}(Y_t)=\e_{\eta}^*g_{\mathcal{D}_t}(\omega)=\sum_{\zeta\subset\eta}\p_{\eta}(\mathcal{D}_t=\zeta)g_{\zeta}(\omega)=\sum_{n=1}^{m}d_{mn}^{\theta}(t)\sum_{|\zeta|=n,\zeta\subset \eta}\mathcal{H}(\zeta,\eta)g_{\zeta}(\omega).
\end{align*}
Then we have the equation (\ref{dp2}) if we replace $g_{\eta}(\omega)=\frac{j_{\eta}(\omega;1)}{\e_{z,z^{\prime},1}j_{\eta}}$ and $g_{\zeta}(\omega)=\frac{j_{\zeta}(\omega;1)}{\e_{z,z^{\prime},1}j_{\zeta}}$.
\end{proof}

\subsection{Proof of Theorem \ref{main} through the dual process method}Next we will use Proposition \ref{dual-prob2} to deduce the transition density (\ref{con}). By Proposition \ref{jack}, we know
$$
\mu_{m}(d\omega)=\sum_{|\eta|=m}\e[j_{\eta}(Y_t;1)\mid Y_0=\omega]\delta_{\frac{\widetilde{\eta}}{m}}(d\omega)
$$
will converge weakly to the distribution of $Y_t$, i.e. the transition probability $q(t,\omega, \cdot)= \p_{\omega}(Y_t\in\cdot)$. By the equation (\ref{dp2}), we know 
\begin{equation}
\mu_{n}(d\omega)=d_{m1}^{\theta}(t)\mu_{1,m}(d\omega)+\sum_{n=2}^{m}d_{mn}^{\theta}(t)\sum_{|\zeta|=n}\mu_{n,m}(d\omega)\frac{j_{\zeta}(\omega;1)}{\e_{z,z^{\prime},1}j_{\zeta}}
\end{equation}
where 
\begin{align*}
\mu_{1,m}(d\omega)=&\sum_{|\eta|=m}[\e_{z,z^{\prime},1}j_{\eta}]\delta_{\frac{\widetilde{\eta}}{m}}(d\omega)\\
\mu_{n,m}(d\omega)=&\sum_{|\eta|=m,\zeta\subset\eta}\mathcal{H}(\eta,\zeta)[\e_{z,z^{\prime},1}j_{\eta}]\delta_{\frac{\widetilde{\eta}}{m}}(d\omega),~n\geq2.
\end{align*}
As $m\to\infty$, we need to show that 
\begin{equation}\label{claim1}
\mu_{1,m}(d\omega)\to \mathcal{Z}(z,z^{\prime},\vartheta)(d\omega)
\end{equation} 
and 
\begin{equation}\label{claim2}
\mu_{n,m}(d\omega)\to j_{\zeta}(\omega;1)\mathcal{Z}(z,z^{\prime},\vartheta)(d\omega)
\end{equation}
The claim (\ref{claim1}) can be directly obtained from Proposition \ref{jack}. Now we are going to show the claim (\ref{claim2}). For any $f\in C(\Omega)$, we have
\begin{align*}
\int_{\Omega}f(\omega)\mu_{n,m}(d\omega)=&\sum_{|\eta|=m,\zeta\subset\eta}f(\frac{\widetilde{\eta}}{m})\mathcal{H}(\eta,\zeta)(\e_{z,z^{\prime},1}j_{\eta})\\
=& \sum_{|\eta|=m,\zeta\subset\eta}f(\frac{\widetilde{\eta}}{m})j_{\zeta}(\frac{\widetilde{\eta}}{m};1)(\e_{z,z^{\prime},1}j_{\eta})+O(\frac{1}{\sqrt{m}})\\
\to& \int_{\Omega}f(\omega)j_{\zeta}(\omega;1)\mathcal{Z}(z,z^{\prime},\vartheta)(d\omega)
\end{align*}
where the second equality is due to Theorem 6.1 and Theorem 7.1 in \cite{KOO}. Since 
$$
\lim_{m\to\infty}d_{mn}^{\theta}(t)=d_{n}^{\theta}(t),
$$
 then we have derived the representation in Theorem \ref{main}.

\section{Further discussion}

\subsection{Conjectures on Diffusions with Given Stationary Measures} The conclusions in this paper indicate that the dual process is only determined by the weights in the branching diagram, whether it be the Kingman graph or the Jack graph. The dual processes are determined by a generator  
$$
\mathcal{L}f(\eta)=-\frac{n(n-1+\theta)}{2}f(\eta)+\frac{n(n-1+\theta)}{2}\sum_{\zeta\subset \eta,|\eta|=|\zeta|+1}p^{\downarrow}(\eta,\zeta)f(\zeta).
$$ 
The diffusions, however, depend on the partition structures on the Kingman graph or the Jack graph. Due to the Kingman's representation theorem and the representation theorem in Proposition \ref{jack}, the partition structures are uniquely determined by their representing measures, which will be the stationary distribution of the diffusions constructed through the up-down Markov chains. 

More generally, for a given probability measure $\mu$ on the Kingman simplex $\overline{\nabla}_{\infty}$, we can consider generator 
$
\mathcal{B}_{\mu}
$
defined on an algebra $\mathcal{A}^{o}_{K}$ spanned by $\{1,\varphi_k,k\geq2\}$ as follows
\begin{equation}\label{conj1}
\mathcal{B}_{\mu}^Kg_{\eta}(x)=-\frac{n(n-1+\theta)}{2}g_{\eta}(x)+\frac{n(n-1+\theta)}{2}\sum_{\zeta\subset \eta,|\eta|=|\zeta|+1}\frac{\dim^K(\zeta)\chi^{K}(\zeta,\eta)}{\dim^K(\eta)}g_{\zeta}(x)
\end{equation}
where $\{g_{\eta}(x)=\frac{m_{\eta}(x)}{\e_{\mu} m_{\eta}}\mid \eta\in\Gamma\}$ is a linear base and $\e_{\mu}$ is the expectation with respect to the probability distribution $\mu$. It will uniquely determine the operation of $\mathcal{B}_{\mu}^K$ on $\mathcal{A}^{o}_{K}$.
\begin{thm}[Conjecture 1]
For a given probability measure $\mu$ on the Kingman simplex $\overline{\nabla}_{\infty}$, the generator defined in (\ref{conj1}) will determine a reversible diffusion $W_t^K$ with the stationary distribution $\mu$. Its transition density is 
$$
p^K(t,x,y)=d_0^{\theta}(t)+d_1^{\theta}(t)+\sum_{n=2}^{\infty}d_n^{\theta}(t)p_n(x,y),
$$
where
$$
p_n(x,y)=\sum_{|\eta|=n}\frac{\dim^K(\eta)m_{\eta}^o(x)\dim^K(\eta)m_{\eta}^o(y)}{\e_{\mu}[\dim^K(\eta)m_{\eta}^o]}.
$$
\end{thm}

Moreover, for a given probability measure $\mu$ in the Thoma simplex $\Omega$, we can also consider generator 
$
\mathcal{B}_{\mu}^{J}
$
defined on an algebra $\mathcal{A}^{o}_{J}$ spanned by $\{1,\varphi_k^{o},k\geq2\}$ as follows
\begin{equation}\label{conj2}
\mathcal{B}_{\mu}^Jg_{\eta}(x)=-\frac{n(n-1+\theta)}{2}g_{\eta}(x)+\frac{n(n-1+\theta)}{2}\sum_{\zeta\subset \eta,|\eta|=|\zeta|+1}\frac{\dim^J_{\vartheta}(\zeta)\chi^{J}_{\vartheta}(\zeta,\eta)}{\dim^J_{\vartheta}(\eta)}g_{\zeta}(x)
\end{equation}
where $\{g_{\eta}(x)=\frac{j_{\eta}(x;\vartheta)}{\e_{\mu} j_{\eta}}\mid \eta\in\Gamma\}$ is a linear base. It will uniquely determine the operation of $\mathcal{B}_{\mu}^J$ on $\mathcal{A}^{o}_{J}$.
\begin{thm}[Conjecture 2]
For a given probability measure $\mu$ on the Thoma simplex $\Omega$, the generator defined in (\ref{conj2}) will determine a reversible diffusion $W_t^J$ with the stationary distribution $\mu$. Its transition density is 
$$
p^J(t,x,y)=d_0^{\theta}(t)+d_1^{\theta}(t)+\sum_{n=2}^{\infty}d_n^{\theta}(t)\mathcal{K}_n(x,y),
$$
where
$$
\mathcal{K}_n(x,y)=\sum_{|\eta|=n}\frac{j_{\eta}(x;\vartheta)j_{\eta}(y;\vartheta)}{\e_{\mu}[j_{\eta}]}.
$$
\end{thm}

\section{Appendix}
\appendix
In this section, we will discuss a few facts about the symmetric functions. Please refer to \cite{Mac} for further details. Symmetric functions are defined as inverse limits of symmetric polynomials. For $n\in\mathbb{Z}_{+}$, let $\Lambda_{n}$ be the ring of symmetric polynomials of variables $x_1,\cdots,x_n$. Define $\rho_{n+1,n}: \Lambda_{n+1}\to\Lambda_{n}$ as follows 
\begin{equation}\label{inver}
\rho_{n+1,n}(f(x_1,\cdots,x_{n},x_{n+1}))=f(x_{1},\cdots,x_{n},0)
\end{equation}
Then $\Lambda$ is the inverse limit of $\Lambda_n,n\geq1$.
\section{Monomial Symmetric Functions} Consider symmetric polynomials 
$$
m_{\eta}^n(x_1,\cdots,x_n)=\sum_{\eta=(a_{(1)},\cdots,a_{(n)})}x_{1}^{a_1}\cdots x_{n}^{a_n}
$$
where $(a_{(1)},\cdots,a_{(n)})$ is the descending arrangement of $(a_1,\cdots, a_n)$. One can see that the equation (\ref{inver}) is also true for $m_{\eta}^n(x_1,\cdots,x_n)$. Then one can define $m_{\eta}(x)$ as an inverse limit of $\{m_{\eta}^n\mid n\geq1\}$, and it is called monomial symmetric function. The evaluation of $m_{\eta}(x)$ in the Kingman simplex $\overline{\nabla}_{\infty}$ can be done through continuous extension of $m_{\eta}|_{\nabla_{\infty}}$. Moreover, $m_{\eta}(x)$ satisfies the Pieri formula
$$
m_{\eta}(x)=\sum_{\eta\subset\nu,|\eta|+1=|\nu|} \chi^{K}(\eta,\nu)m_{\nu}(x).
$$
and $1=\sum_{|\eta|=n}\dim^{K}(\eta)m_{\eta}(x)$.
\section{Shur Functions} Denote $S_n$ as symmetric group. Consider symmetric polynomials
$$
s_{\eta}^n(x_1,\cdots,x_n)=\frac{a_{\eta+\delta}^n(x_1,\cdots,x_n)}{a_{\delta}(x_1,\cdots,x_n)},
$$
where $\eta=(\eta_1,\cdots,\eta_n)$ and $\delta=(n-1,n-2,\cdots,1,0)$ are integer partitions, and 
$$
a_{\eta}(x_1,\cdots,x_n)=\sum_{\sigma\in S_n}\mathrm{Sgn}(\sigma)x_{\sigma(1)}^{\eta_1}\cdots x_{\sigma(n)}^{\eta_n}.
$$ 
One can also show that $\{s^n_{\eta}\mid n\geq1\}$ satisfy equation (\ref{inver}). Then the Shur function $s_{\eta}(x)$ is defined to be the inverse limit of $\{s^n_{\eta}\mid n\geq1\}$. Moreover, the Shur functions also satisfies the Pieri formula 
$$
s_{\eta}(x)=\sum_{\eta\subset\nu,|\eta|+1=|\nu|} \chi^{J}_{1}(\eta,\nu)s_{\nu}(x).
$$
and $1=\sum_{|\eta|=n}\dim^{J}_{1}(\eta)s_{\eta}(x)$, where $\dim^{J}_{1}=\frac{n!}{H(\eta;1)}$.

\section{Jack Functions} Jack polynomials $J_{\lambda}^n(x_1,\cdots,x_n;\vartheta)$ are defined to be the eigenfunctions of the Sekiguchi operators:
  \begin{align*}
  D(u;\vartheta)=&\frac{1}{\prod_{i<j}(x_i-x_j)}\det\left[x_i^{n-j}(x_i\frac{\partial}{\partial x_i}+(n-j)\vartheta+u)\right]_{1\leq i,j\leq n}\\
  D(u;\vartheta)J_{\lambda}^n(x;\vartheta)=&\left[\prod_{i=1}^{n}(\lambda_i+(n-i)\vartheta+u)\right]J_{\lambda}^n(x;\vartheta)
  \end{align*}
  One can also show that $\{J^n_{\eta}(x;\vartheta)\mid n\geq1\}$ satisfy the equation (\ref{inver}). Then the Jack symmetric function $J_{\lambda}(x_1,\cdots,x_n,\cdots;\vartheta)$ is defined to be the inverse limit of $\{J^n_{\eta}(x;\vartheta)\mid n\geq1\}$.  Moreover, the Jack functions satisfy the Pieri formula 
  $$
J_{\eta}(x;\vartheta)=\sum_{\eta\subset\nu,|\eta|+1=|\nu|} \chi^{J}_{\vartheta}(\eta,\nu)J_{\nu}(x;\vartheta).
$$
and
  $$
    \left(\sum_{i=1}^{\infty}x_i\right)^{n}=\sum_{|\eta|=n}\dim_{\vartheta}^{J}(\eta)J_{\eta}(x;\vartheta).
  $$

In particular, when $\vartheta=1$, $J_{\eta}(x;1)=s_{\eta}(x)$.

\end{document}